\documentclass[graybox]{svmult}

\usepackage[pagewise]{lineno}   

\usepackage{mathptmx}       
\usepackage{helvet}         
\usepackage{courier}        
\usepackage{type1cm}        
\usepackage{makeidx}         
\usepackage{graphicx}        
\usepackage{multicol}        
\usepackage[bottom]{footmisc}

\usepackage{pstricks-add}
\usepackage{verbatim}
\usepackage{amssymb}
\usepackage{latexsym}
\usepackage[all]{xy}

\newcommand{\pp}{\mathcal{P}}

\newrgbcolor{fffftt}{1 1 0.2}
\newrgbcolor{ffttqq}{1 0.2 0}
\newrgbcolor{ffffzz}{1 1 0.6}
\newrgbcolor{qqcctt}{0 0.8 0.2}

\makeindex             
 
\begin{document}


\title*{Hereditary Polytopes}

\author{Mark Mixer, Egon Schulte and Asia Ivi\'{c} Weiss
\\[.3in]
{\it With best wishes to our friend and colleague Peter McMullen.}
}
\institute{Mark Mixer, \at York University, Department of Mathematics and Statistics, Toronto, Ontario M3J 1P3, Canada, \email{mark.mixer@gmail.com}
\and Egon Schulte, \at Northeastern University, Department of Mathematics, Boston, MA 02115, USA, \email{schulte@neu.edu}
\and Asia Ivi\'{c} Weiss, \at York University, Department of Mathematics and Statistics, Toronto, Ontario M3J 1P3, Canada, \email{weiss@mathstat.yorku.ca}}
%
%
\maketitle

\abstract{Every regular polytope has the remarkable property that it inherits all symmetries of each of its facets. This property distinguishes a natural class of polytopes which are called hereditary.  Regular polytopes are by definition hereditary, but the other polytopes in this class are interesting, have possible applications in modeling of structures, and have not been previously investigated.  This paper establishes the basic theory of hereditary polytopes, focussing on the analysis and construction of hereditary polytopes with highly symmetric faces.}

\section{Introduction}
\label{intro}

In the classical theory of convex polyhedra,  the Platonic and Archimedean solids form a natural class of highly symmetric objects.  The symmetry group of each of these polyhedra acts transitively on its vertices.  If we restrict to those solids whose symmetry groups also act transitively on their edges, only the regular polyhedra, the cuboctahedron, and the icosidodecahedron remain.  These polytopes all have the distinguishing property that every symmetry of their polygonal faces extends to a symmetry of the solid.  In fact, if we look for convex ``hereditary'' polyhedra (those having this property of inheriting all the symmetries of their faces) with regular faces, we find that vertex and edge transitivity is implied (as we shall see in a more general setting in Section 4).  

It is natural to generalize this idea of hereditary polyhedra to the setting of abstract polytopes of any rank.  In this paper we study those polytopes that have the property of inheriting all symmetries of their facets.  The formal definition of a hereditary polytope can be found in Section 2, along with other basic notions required for the understanding of this paper.

An abstract polytope of rank 3 can be seen as a map, that is a 2-cell embedding of a connected graph into a closed surface.   Regular and chiral maps have been studied extensively in the past (see for example~\cite{Conder}, \cite{cm}), and form a natural class of highly symmetric maps.  In some older literature, chiral maps are labeled as regular, as locally they are regular in the following sense.  The symmetry group of a chiral map acts transitively on the vertices, edges, and faces, and the maps have the maximal possible rotational symmetry.  However, none of the reflectional symmetry of any of the faces of a chiral map extends to a global symmetry.   Therefore chiral maps, although highly symmetric, are not hereditary in our sense.  

The non-regular hereditary maps are the 2-orbit maps which are vertex and edge transitive.  This type of map has been extensively studied (see for example~\cite{gw}, \cite{hub},\cite{stw}).  It will be shown that certain 2-orbit polytopes will always be hereditary (see Theorem~\ref{2orbit}).  However, the characterization of hereditary polytopes of rank greater than three is complex.

In Section 3, we consider how various transitivity properties of the facets affect the transitivity properties of a hereditary polytope.  Section 4 deals with polytopes with regular facets, with an emphasis on hereditary polyhedra.  In Section 5, we consider polytopes with chiral facets, and prove the existence of certain hereditary polytopes of this type.  In Section 6, some questions regarding the extensions of hereditary polytopes are considered.  We conclude with a brief section which suggests some interesting problems for related work.

\section{Basic notions}
\label{bano}

Following \cite{arp}, a \textit{polytope} $\pp$ of rank $n$, or an \textit{$n$-polytope}, is a partially ordered set of faces, with a strictly monotone rank function having range $\{-1, \ldots , n\}$. The elements of $\pp$ with rank $j$ are called \textit{$j$-faces}; typically $F_j$ indicates a $j$-face.  A \textit{chain} of type $\{i_1,\ldots,i_k\}$ is a totally ordered set faces of ranks $\{i_1,\ldots,i_k\}$.  The maximal chains of $\pp$ are called flags.   We require that $\pp$ have a smallest $(-1)$-face $F_{-1}$, a greatest $n$-face $F_n$ and that each flag contains exactly $n+2$ faces.  Also, $\pp$ should be \textit{strongly flag-connected}, that is, any two flags $\Phi$ and $\Psi$ of $\pp$ can be joined by a sequence of flags $\Phi:=\Phi_0,\Phi_1,\ldots,\Phi_k=:\Psi$ such that each $\Phi_{i}$ and $\Phi_{i+1}$ are \textit{adjacent} (in the sense that they differ by just one face), and $\Phi \cap \Psi \subseteq \Phi_i$ for each $i$. Furthermore, we require the following homogeneity property: whenever $F <G$, with rank$(F) = j-1$ and rank$(G) = j+1$, then there are exactly two $j$-faces $H$ with $F < H < G$. Essentially, these conditions say that $\pp$ shares many combinatorial properties with the face lattice of a convex polytope.

If $\Phi$ is a flag, then we denote the \textit{i-adjacent flag} by $\Phi^i$, that is the unique flag adjacent to $\Phi$ and differing from it in the face of rank $i$.  More generally, we define inductively $\Phi^{i_1,\ldots,i_k} := (\Phi^{i_1,\ldots,i_{k-1}})^{i_k}$ for $k \geq 2$.   Note that if $ | i  - j | \geq 2$, then $\Phi^{i,j} =\Phi^{j,i}$; otherwise in general, $\Phi^{i,j} \neq \Phi^{j,i}$.

The faces of rank 0, 1, and $n - 1$ are called \textit{vertices, edges, and facets}, respectively.  We will sometimes identify a facet $F_{n-1}$ with the section $F_{n-1}/F_{-1}$ when there is no chance for confusion.  If $F$ is a vertex, the section $F_n/F := \{G|F \leq G \leq F_n\}$ is called the \textit{vertex-figure} of $\pp$ at $F$.   A polytope is said to be \textit{equivelar} of (Schl\"afli) \textit{type} $\{p_1,\ldots,p_{n-1}\}$ if each section $F_{i+1} / F_{i-2}$ is combinatorially equivalent to a $p_i$-gon.  Additionally, if the facets of $\pp$ are all isomorphic to an $(n-1)$-polytope $\mathcal{K}$ and its vertex-figures are all isomorphic to an $(n-1)$-polytope $\mathcal{L}$, then we say $\pp$ is of type $\{\mathcal{K},\mathcal{L}\}$. (This is a small change of terminology from~\cite{arp}.)

The set of all automorphisms of $\pp$ is a group denoted by $\Gamma(\pp)$ and called the \textit{automorphism group} of $\pp$.   For $0 \leq i \leq n-1$, an $n$-polytope $\pp$ is called {\em $i$-face transitive\/} if $\Gamma(\pp)$ acts transitively on the set of $i$-faces of $\pp$. In addition, $\pp$ is said to be {\em $\{0,1,\ldots,i\}$-chain transitive\/} if $\Gamma(\pp)$ acts transitively on the set of chains of $\pp$ of type $\{0,1,\ldots,i\}$.

A polytope $\pp$ is said to be \textit{regular} if $\Gamma(\pp)$ acts transitively on the flags, that is if $\pp$ is $\{0,1,\ldots,n-1\}$-chain transitive.   The automorphism group of a regular $n$-polytope is known to be a \textit{string C-group} (a smooth quotient of a Coxeter group with a linear diagram, which satisfies a specified intersection condition), and is generated by involutions $\rho_0,\ldots,\rho_{n-1}$, which are called the \textit{distinguished generators} associated with a flag $\Phi$, and defined as follows.  Each $\rho_i$ maps $\Phi$ to $\Phi^i$.  These generators for a polytope of Schl\"afli type $\{p_1,\ldots,p_{n-1}\}$ satisfy relations of the form 

\begin{equation}
( \rho_i \rho_j ) ^ {p_{ij}} = \varepsilon \textrm{ for } i,j = 0,\ldots,n-1,
\end{equation} 

\noindent where $p_{ii} =1$, $p_{ij}=p_{ji}:=p_{j}$ if $j=i+1$, and $p_{ij} = 2$ otherwise.  When the sections $F_{n-1} / F_{-1}$ of a polytope $\pp$ are themselves regular, we way that $\pp$ is \textit{regular-facetted}. 

A polytope $\pp$ is said to be \textit{chiral} if there are two orbits of flags under the action of $\Gamma(\pp)$ and adjacent flags are in different orbits.  In this case, given a flag 
$\Phi = \{F_{-1}, \ldots , F_n\}$ of $\pp$ there exist automorphisms, which are also called distinguished generators, $\sigma_1, \ldots , \sigma_{n-1}$ of $\pp$ such that each $\sigma_i$ fixes all faces in $\Phi \setminus \{F_{i -1} , F_i \}$ and cyclically permutes consecutive $i$-faces of $\pp$ in the rank 2 section of $F_{i+1}/F_{i-2}$.  Each chiral polytope comes in two \textit{enantiomorphic forms}; one associated with a base flag $\Phi$ and the other with any of its adjacent flags.    When the sections $F_{n-1} / F_{-1}$ of a polytope $\pp$ are themselves chiral, we say that $\pp$ is \textit{chiral-facetted}.

A polytope $\pp$ is said to be \textit{$k$-orbit} if there are $k$ orbits of flags under the action of $\Gamma(\pp)$.  In the case of 2-orbit polytopes, if $i$-adjacent flags are in the same orbit for $i \in I$ and in different orbits for $i \not \in I$, then we say that $\pp$ is in the \textit{class} $2_I$.

Finally, a polytope $\pp$ is called {\em hereditary\/} if for each facet $F$ of $\pp$ the group $\Gamma(F/F_{-1})$ of the corresponding section $F/F_{-1}$ is a subgroup of $\Gamma(\pp)$; in fact, then $\Gamma(F/F_{-1})$ is a subgroup of $\Gamma_{F}(\pp)$, the stabilizer of $F$ in $\Gamma(\pp)$. More informally, $\pp$ is hereditary if every automorphism of every facet $F$ extends to an automorphism of $\pp$ which fixes $F$.

\section{Transitivity on faces}
\label{trafa}

We begin with a number of basic properties of hereditary polytopes which have highly-symmetric facets.

\begin{proposition}\label{ichain}
If an $n$-polytope $\pp$ is hereditary and each facet is $\{0,1,\ldots,i\}$-chain transitive for some $i$ with $i \leq n-2$, then $\pp$ is $\{0,1,\ldots,i\}$-chain transitive, and hence the $i$-faces of $\pp$ are mutually isomorphic regular $i$-polytopes.  
\end{proposition}

\begin{proof}
Let $\Phi$ and $\Psi$ be two chains of $\pp$ of type $\{0,1,\ldots,i\}$. Since $\pp$ is strongly flag-connected and $i\leq n-2$, there exists a sequence $\Phi:=\Phi_0,\Phi_1,\ldots,\Phi_k:=\Psi$ of chains of type $\{0,1,\ldots,i\}$ such that, for $j=1,\ldots,k$, all faces of $\Phi_{j-1}$ and ${\Phi_j}$ are incident to a common facet $H_j$.  As each facet is transitive on chains of this type, there is an automorphism of $H_j$ mapping $\Phi_{j-1}$ to $\Phi_j$.  These automorphisms of the facets $H_j$ are also automorphisms of $\pp$, and their composition maps $\Phi$ to $\Psi$.
\qed\end{proof}

In much the same way we can also prove the following proposition, again appealing to the strong flag-connectedness. 

\begin{proposition}\label{iface}
If an $n$-polytope $\pp$ is hereditary and each facet is $i$-face transitive for some $i$ with $i \leq n-2$, then $\pp$ is $i$-face transitive. In particular, if each facet is vertex transitive, then $\pp$ is vertex transitive.
\end{proposition}

\begin{proof}
Join any two $i$-faces of $\pp$ by a sequence of $i$-faces in which any two consecutive $i$-faces lie in a common facet. Then  proceed as in the previous proof.
\qed\end{proof}

Proposition~\ref{ichain} also has the following immediate consequence.

\begin{proposition}\label{corollary}
If an $n$-polytope $\pp$ is hereditary and each facet is regular, then the $(n-2)$-faces of $\pp$ are all regular $(n-2)$-polytopes mutually isomorphic under isomorphisms induced by automorphisms of $\pp$.
\end{proposition}

Our main interest is in hereditary polytopes all of whose facets are either regular or chiral. The following theorem says that any such polytope must have its facets either all regular or all chiral. In other words, the ``mixed-case" does not occur. 

\begin{theorem}
\label{nomix}
If $\pp$ is a hereditary polytope with each facet either regular or chiral, then either each facet of $\pp$ is regular or each facet of $\pp$ is chiral.
\end{theorem}

\begin{proof}
Suppose $\pp$ has at least one regular facet. We prove that then each facet of $\pp$ must be regular. By the connectedness properties of $\pp$ it suffices to show that each facet adjacent to a regular facet must itself be regular.

Let $H$ be a regular facet, and let $H'$ be an adjacent facet meeting $H$ in an $(n-2)$-face~$G$. Let $\Omega$ be a flag of $H/F_{-1}$ containing $G$. Since $H$ is regular, its group $\Gamma(H/F_{-1})$ contains a ``reflection" $\rho_{0}^H$ which maps $\Omega$ to $\Omega^0$, the $0$-adjacent flag of $\Omega$ in $H/F_{-1}$. Since $\pp$ is hereditary, $\rho_{0}^H$ extends to an automorphism of $\pp$, again denoted $\rho_{0}^H$, which takes the flag $\Psi:=\Omega\cup \{F_n\}$ of $\pp$ to $\Psi^0$. But $\rho_{0}^H$ fixes $H$ and $G$, so must necessarily fix $H'$ as well and hence belong to $\Gamma(H'/F_{-1})$. Moreover, $\rho_{0}^H$ maps the flag $\Omega':=(\Omega\setminus \{H\})\cup \{H'\}$ of $H'/F_{-1}$ to its $0$-adjacent flag~$(\Omega')^0$. Thus $\Gamma(H'/F_{-1})$ contains an element which takes a flag of $H'/F_{-1}$ to an adjacent flag. On the other hand, each facet of $\pp$ is regular or chiral, so $H'$ must necessarily be regular. Bear in mind here that a chiral polytope does not admit an automorphism mapping a flag to an adjacent flag.
\qed\end{proof}

A hereditary polytope with some of its facets regular, need not have all of its facets regular. This is illustrated by the example of the semiregular tessellation $\mathcal{T}$ of Euclidean $3$-space by regular tetrahedra and (vertex) truncated regular tetrahedra. This tessellation is related to the Petrie-Coxeter polyhedron $\{6,6\,|\,3\}$. The facets (tiles) of $\mathcal{T}$ are of two kinds, namely (regular) Platonic solids and (semiregular) Archimedean solids, or more precisely, truncated Platonic solids. This tessellation is a hereditary $4$-orbit polytope of rank $4$. 

More examples arise in a similar way from the semiregular tessellations of the $3$-sphere $\mathbb{S}^3$ or hyperbolic $3$-space $\mathbb{H}^3$ related to the regular skew polyhedra $\{2l,2m\,|\,r\}$ in these spaces. Their tiles are Platonic solids $\{r,m\}$ and (vertex) truncated Platonic solids $\{l,r\}$. These tessellations can be derived by Wythoff's construction applied to the spherical or hyperbolic $4$-generator Coxeter group with square diagram  
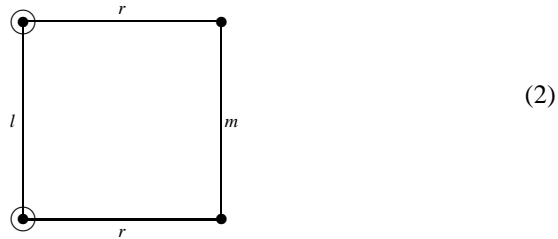
\begin{equation}
\label{diag1}
\centering
\begin{picture}(100,40)
\put(0,-44){
\multiput(15,0)(75,0){2}{\circle*{4}}
\multiput(15,75)(75,0){2}{\circle*{4}}
\put(15,0){\circle{9}}
\put(15,75){\circle{9}}
\put(15,0){\line(1,0){75}}
\put(15,75){\line(1,0){75}}
\put(15,0){\line(0,1){75}}
\put(90,0){\line(0,1){75}}
\put(10,35){\scriptsize{$l$}}
\put(91.3,35){\scriptsize{$m$}}
\put(51,-6.5){\scriptsize{$r$}}
\put(51,78){\scriptsize{$r$}}}
\end{picture}
\end{equation}
\vskip.66in\noindent
More details, including a list of the various possible choices for $l,m,r$, can be found in~\cite{crsp,luwe,spacfi}. The semiregular tessellation $\mathcal{T}$ in $\mathbb{E}^3$ mentioned earlier is obtained in a similar fashion from the Euclidean Coxeter group given by the diagram in (\ref{diag1}) with $l=m=r=3$.

\section{Hereditary polytopes with regular facets}
\label{herreg}

In this section we investigate hereditary polytopes which are regular-facetted. We show that each such polytope is either itself regular or a $2$-orbit polytope.

\subsection{Flag-orbits}

We begin with the following observation.

\begin{proposition}\label{oddq}
Let $\pp$ be a regular-facetted hereditary $n$-polytope. If there exists an $(n-3)$-face $F$ such that its co-face $F_{n}/ F$ is a $q$-gon with $q$ odd, then $\pp$ is a regular $n$-polytope of Schl\"{a}fli type $\{p_1,\ldots,p_{n-2},q\}$, where $\{p_1,\ldots,p_{n-2}\}$ is the Schl\"afli type of any facet of $\pp$.
\end{proposition}

\begin{proof}
The proof exploits the fact that for odd $q$ the dihedral group $D_q$ has just one conjugacy class of reflections. Geometrically speaking this means that the reflection mirror bisecting an edge of a convex regular $q$-gon also bisects the angle at the opposite vertex. This conjugacy class then generates $D_q$.

First observe that, by Proposition~\ref{iface}, the group $\Gamma(\pp)$ is transitive on the $(n-3)$-faces of $\pp$ since $\pp$ has regular facets. (This already implies that {\em each\/} co-face of an $(n-3)$-face is a $q$-gon.)  Now, if we can show that the stabilizer of an $(n-3)$-face in $\Gamma(\pp)$ acts transitively on the flags of $\pp$ containing that $(n-3)$-face, then clearly $\Gamma(\pp)$ acts flag-transitively on $\pp$ and hence $\pp$ must be regular.

Now suppose $F$ is an $(n-3)$-face such that $F_{n}/ F$ is a $q$-gon. Clearly, since the facets of $\pp$ are regular, $F$ is also regular and its group $\Gamma(F/F_{-1})$ is a subgroup of the automorphism group of any facet of $\mathcal{P}$ that contains $F$. Moreover, since $\pp$ is hereditary, $\Gamma(F/F_{-1})$ is also a subgroup of $\Gamma(\pp)$ acting flag-transitively on $F/F_{-1}$ and trivially on $F_{n}/F$. 

On the other hand, if $H$ is any facet of $\pp$ containing $F$, then there exists a unique involution $\rho_{n-2}^H$ (say) in $\Gamma(H/F_{-1})$ which fixes a flag of $F/F_{-1}$ and interchanges the two $(n-2)$-faces of $H$ containing $F$. Now, since $q$ is odd, the reflections $\rho_{n-2}^H$, with $H$ a facet containing $F$, generate a subgroup isomorphic to the dihedral group $D_q$. Hence, since this subgroup acts flag-transitively on $F_{n}/F$ and trivially on $F/F_{-1}$, it can be identified with $\Gamma(F_{n}/F)$. 

Then $\Gamma_{F}(\pp)=\Gamma(F/F_{-1})\times\Gamma(F_{n}/F)$, and $\Gamma_{F}(\pp)$ acts transitively on the flags of $\pp$ that contain $F$.
\qed\end{proof}

The following theorem says that the hereditary polytopes with regular facets fall into two classes.

\begin{theorem}\label{2orbit}
A regular-facetted $n$-polytope is hereditary if and only if it is regular or a 2-orbit polytope in the class  $2_{\{0,1,\ldots,n-2\}}$.
\end{theorem}

\begin{proof}
Let $\pp$ be a regular-facetted hereditary $n$-polytope. As before, $\pp$ is $(n-3)$-face transitive. Let $F$ be any $(n-3)$-face of $\pp$. We must show that the stabilizer $\Gamma_{F}(\pp)$ has at most two orbits on the set of flags of $\pp$ containing $F$. Since $F$ is regular and $\pp$ is hereditary, $\Gamma(F/F_{-1})$ is again a subgroup of $\Gamma_F(\pp)$ acting flag-transitively on $F/F_{-1}$ and trivially on $F_{n}/F$. 

As in the previous proof, for each facet $H$ of $\pp$ containing $F$, there exists a unique involution $\rho_{n-2}^H$ (say) in $\Gamma(H/F_{-1})$ which fixes a flag of $F/F_{-1}$ and interchanges the two $(n-2)$-faces of $H$ containing $F$. Suppose the co-face $F_{n}/F$ is a $q$-gon, allowing $q=\infty$. By Proposition~\ref{oddq}, if $q$ is odd, then $\pp$ is regular and we are done. 

Now suppose $\pp$ is not regular. Then $q$ is even or $q=\infty$. In this case the subgroup $\Lambda$ of $\Gamma_{F}(\pp)$ generated by the involutions $\rho_{n-2}^H$, with $H$ a facet containing $F$, is isomorphic to a dihedral group $D_{q/2}$; when restricted to the $q$-gonal co-face $F_{n}/F$, these  involutions $\rho_{n-2}^H$ generate a subgroup of index $2$ in the full dihedral automorphism group $D_q$ of $F_{n}/F$. Hence $\Lambda$, restricted to $F_{n}/F$, has two flag-orbits on $F_{n}/F$. It follows that $\Gamma_{F}(\pp)=\Gamma(F/F_{-1})\times\Lambda$, and that $\Gamma_{F}(\pp)$ has two orbits on the flags of $\pp$ that contain $F$. Thus $\pp$ is a $2$-orbit polytope. Moreover, since $\pp$ is hereditary and the facets of $\pp$ are regular, 
$\Gamma(\mathcal{P})$ contains the distinguished generators for the group of any facet of $\pp$, so $\pp$ is necessarily of type $2_I$ with $\{0,\ldots,n-2\}\subseteq I$. On the other hand, since $\pp$ itself is not regular, no flag can be mapped onto its $(n-1)$-adjacent flag by an automorphism of $\pp$. Hence $\pp$ must be a $2$-orbit polytope in the class $2_{\{0,1,\ldots,n-2\}}$. 

Conversely, if $\pp$ is in the class  $2_{\{0,1,\ldots,n-2\}}$, then it has regular facets, and since all flags that contain a mutual facet are in the same orbit, it is hereditary.

\qed\end{proof}

Note that every $2$-orbit polytope $\pp$ in the class $2_{\{0,1,\ldots,n-2\}}$ necessarily has regular facets, generally of two different kinds. In particular, $\pp$ has a generalized Schl\"afli symbol of the form
\[ \Big\{p_1, \dots, p_{n-3}, \begin{array}{l} p_{n-2} \\ q_{n-2} \end{array}\Big\},\]
where $\{p_1,\dots, p_{n-3},p_{n-2}\}$ and $\{p_1,\dots, p_{n-3},q_{n-2}\}$ are the ordinary Schl\"afli symbols for the two kinds of facets (see \cite{hubsch}). This is a generalization of the classical Schl\"afli symbol used in Coxeter~\cite{cox2} for semiregular convex polytopes.  
\smallskip

We now describe some examples of regular-facetted hereditary polytopes of low rank, concentrating mainly on rank $3$. All regular polytopes are hereditary and (trivially) regular-facetted, so we consider only non-regular polytopes, which, as we just proved, must necessarily be $2$-orbit polytopes in the class $2_{\{0,1,\ldots,n-2\}}$.

\subsection{Hereditary polyhedra}
\label{hpolyhedra}

Since all abstract 2-polytopes (polygons) are regular, by Theorem~\ref{2orbit}, each hereditary polyhedron is (trivially) regular-facetted and hence is either regular or a $2$-orbit polyhedron in the class $2_{\{0,1\}}$. Both the cuboctahedron and the icosidodecahedron can easily be seen to be hereditary polyhedra. In fact, these are the only hereditary polyhedra amongst the Archimedean solids. Similarly, the uniform Euclidean plane tessellation of type $(3.6)^2$ is an infinite hereditary polyhedron (see \cite{grsh}). 

Recall that the {\em medial\/} of a polyhedron (map) $\pp$ on a closed surface is the polyhedron $\rm{Me}(\pp)$ on the same surface whose vertices are the ``midpoints" of the edges of $\pp$, and whose edges join two vertices if and only if the corresponding edges of $\pp$ are adjacent edges of a face of $\pp$. All three examples of hereditary polyhedra just mentioned can be constructed as medials of regular maps, namely of the cube $\{4,3\}$, the dodecahedron $\{5,3\}$, and the euclidean plane tessellation $\{6,3\}$, respectively. 

More generally, given a regular polyhedron $\pp$ of type $\{p,q\}$, the medial $\rm{Me}(\pp)$ is a hereditary polyhedron, and $\rm{Me}(\pp)$ is regular if and only if $\pp$ is self-dual (see \cite[Theorem~4.1]{OPW}). This can be quickly seen algebraically. If the automorphism group of the original polyhedron is $\Gamma(\pp) = \langle \rho_0,\rho_1,\rho_2 \rangle$ (say), then $\Gamma(\rm{Me}(\pp))$ is isomorphic to $\Gamma(\pp)$ if $\pp$ is not self-dual, or $\Gamma(\pp) \ltimes C_2$ if $\pp$ is self-dual (this latter group is just the extended group of $\pp$, consisting of all automorphisms and dualities of $\pp$). When $\pp$ is not self-dual, there are generally two kinds of facets of $\rm{Me}(\pp)$, namely $p$-gons corresponding to conjugate subgroups of $\langle \rho_0, \rho_1 \rangle$ in $\Gamma(\pp)$, and $q$-gons corresponding to conjugate subgroups of $\langle \rho_1,\rho_2 \rangle$ in $\Gamma(\pp)$; in particular, when $q=p$ all facets of $\rm{Me}(\pp)$ are $p$-gons, so $\rm{Me}(\pp)$ has facets of just one type (we describe an example below). This is also true when $\pp$ is self-dual; however, in this case the two subgroups are conjugate in the extended group of $\pp$ (under a polarity fixing the base flag). In either case, $\rm{Me}(\pp)$ is hereditary since the two kinds of conjugate subgroups in $\Gamma(\pp)$ are also subgroups of $\Gamma(\rm{Me}(\pp))$.

Using the medial construction, we can find a finite hereditary polyhedron with only one isomorphism type of facet, which, although it has a Schl\"afli symbol, is not regular.  Consider a non self-dual polyhedron of type $\{p,p\}$, for example the polyhedron $\pp$ of type $\{5,5\}_{12}$ denoted as ``N98.6" by Conder~\cite{Conder} (or as $\{5,5\}*1920b$ by Hartley~\cite{Hartley}). The medial of $\pp$ is a polyhedron of type $\{5,4\}$ with the same automorphism group, of order $1920$, but with twice as many flags. Thus this polyhedron is not regular, but it is still hereditary and of type $2_{\{0,1\}}$.

The previous example is also of independent interest with regards to the following problem about the lengths of certain distinguished paths in its edge graph.  

\begin{remark}\label{rem}
In Problem~7 of~\cite{banff}, it is asked to what extent a regular or chiral polyhedron of type $\{p,q\}$ is determined by the lengths of its $j$-holes and the lengths of its $j$-zigzags.  The polyhedron $\mathcal{P}$ with 1920 flags, mentioned above, has Petrie polygons (1-zigzags) of length 12,\ 2-zigzags of length 5, and 2-holes of length 12.  Thus, we say it is of type $\{5,5\,|\, 12 \}_{12,5}$ (see \cite[p.\,196]{arp}).  However, calculation in \textsc{Magma}~\cite{magma} shows that the universal polyhedron of this type has 30720 flags. This gives an example of a regular polyhedron which is not determined by the lengths of all of its j-holes and the lengths of all of its j-zigzags.
\end{remark}

Every non-regular hereditary polyhedron $\pp$, by Proposition~\ref{oddq}, has vertex-figures which are $2q$-gons for some $q$. In particular, by Theorem~\ref{2orbit}, $\pp$ is a $2$-orbit polyhedron in class $2_{\{0,1\}}$.  Additionally, Theorem 4.2 of~\cite{OPW} shows that any $2$-orbit polyhedron in class $2_{\{0,1\}}$ is the medial of a regular map if and only if $q=2$.  

However, there are non-regular hereditary polyhedra which are not medials of regular maps.  We now define a class of such examples via a map operation which we call ``generalized halving."  The halving operation itself is described in Section~\ref{halv}.  If $\mathcal{K}$ is a regular map of type $\{2p,q\}$ whose edge graph is bipartite, then we define a hereditary polyhedron $\mathcal{K}^a$ (on the same surface as $\mathcal{K}$) as follows; here ``a" stands for ``alternating vertices" (see also Section~\ref{halv}).   Suppose that the vertices of $\mathcal{K}$ are colored {\em red\/} and {\em yellow\/} such that adjacent vertices have different colors.  The vertex-figures at the red vertices of $\mathcal{K}$ (obtained by joining the yellow vertices adjacent to a given red vertex in cyclic order) form one class of facets of $\mathcal{K}^a$.  The other class of facets of $\mathcal{K}^a$ is defined by joining the yellow vertices of a facet of $\mathcal{K}$ whenever they are adjacent to the same red vertex in that facet. The resulting polyhedron has facets of type $\{p\}$ and $\{q\}$, and vertex-figures of type $\{2q\}$.  The polyhedron $\mathcal{K}^a$ is in the class $2_{\{0,1\}}$, and thus is hereditary.

Hereditary polyhedra can be seen as quotients of the uniform tessellations $(p.q)^r$ of the sphere, Euclidean plane, or hyperbolic plane, which can be derived by Wythoff's construction from the $(p,q,r)$ extended triangle group as indicated below (see~\cite{cox2}).

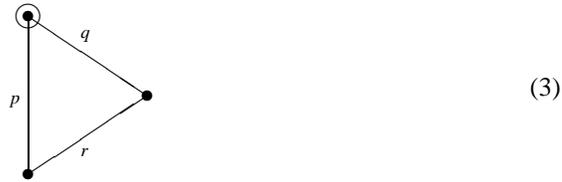
\begin{equation}
\label{diag2}
\centering
\begin{picture}(100,40)
\put(0,-30){
\put(15,0){\circle*{4}}
\put(15,60){\circle*{4}}
\put(60,30){\circle*{4}}
\put(15,0){\line(0,1){60}}
\put(15,60){\circle{9}}
\put(15,0){\line(3,2){45}}
\put(15,60){\line(3,-2){45}}
\put(8,27){\scriptsize{$p$}}
\put(35,52){\scriptsize{$q$}}
\put(35,7){\scriptsize{$r$}}
}
\end{picture}
\end{equation}
\vskip.5in\noindent 

\noindent In particular, the hereditary polyhedra arising as medials of regular maps of type $\{p,q\}$ are quotients of the above infinite tessellations constructed from the $(p, q, 2)$ extended triangle groups.  Similarly, the polyhedra arising from our generalized halving construction of a regular map of type $\{2p,q\}$ are quotients of the infinite tessellations constructed from the $(p, q, q)$ extended triangle groups.

Moving on to rank $4$, we observe that the semi-regular tessellation of Euclidean 3-space by regular tetrahedra and octahedra gives a simple example of a regular-facetted hereditary polytope which is not regular. Its geometric symmetry group is a subgroup of index $2$ in the symmetry group of the cubical tessellations of $3$-space. Note that the combinatorial automorphism group of either tessellation is isomorphic to its symmetry group. 

\section{Hereditary polytopes with chiral facets}
\label{herchir}

When a hereditary polytope has chiral facets, its rank is at least $4$. In this section we show that any such polytope has either two or four flag-orbits. 

\subsection{Flag-orbits}
\label{flobfc}

Call an abstract polytope $\mathcal{P}$ {\em equifacetted\/} if its facets are mutually isomorphic. All regular or chiral polytopes are equifacetted. A $2$-orbit $n$-polytope in a class $2_I$ with $n-1\in I$ is also equifacetted. 

\begin{theorem}
\label{chirfac}
A chiral-facetted hereditary $n$-polytope is a $2$-orbit polytope which is either chiral or in class $2_{\{n-1\}}$ (and hence is equifacetted), or a $4$-orbit polytope. 
\end{theorem}

\begin{proof}
Let $\pp$ be a chiral-facetted hereditary $n$-polytope. First note that we must have $n\geq 4$, since the facets of polytopes of rank at most $3$ are always regular, not chiral. By Proposition~\ref{iface}, the polytope $\pp$ is $(n-3)$-face transitive since its facets are $(n-3)$-face transitive. In particular, any flag of $\pp$ is equivalent under $\Gamma(\pp)$ to a flag containing a fixed $(n-3)$-face $F$ of $\pp$. Again we employ the action of the stabilizer $\Gamma_F(\pp)$ on the set of flags of $\pp$ containing $F$. 

Let $F$ be an $(n-3)$-face of $\pp$, and let $\Omega$ be a flag of the section $F/F_{-1}$. For each facet $H$ of $\pp$ containing $F$ there exists a unique involution $\tau_{0,n-2}^H$ (say) in $\Gamma(H/F_{-1})$ which interchanges the two $(n-2)$-faces of $H$ containing $F$ while fixing all faces of $\Omega$ except the $0$-face. Let $\Lambda$ denote the subgroup of $\Gamma_F(\pp)$ generated by the involutions $\tau_{0,n-2}^H$, with $H$ a facet containing $F$. Now suppose again that the $2$-polytope $F_{n}/F$ is a $q$-gon, allowing $q=\infty$. When restricted to the co-face $F_{n}/F$, the involutions $\tau_{0,n-2}^H$ act like reflections in perpendicular bisectors of edges of a convex regular $q$-gon, and so the restricted $\Lambda$ is isomorphic to a dihedral group $D_q$ or $D_{q/2}$ according as $q$ is odd or even. Hence~$\Lambda$, restricted to $F_{n}/F$, has one or two flag-orbits on the $2$-polytope $F_{n}/F$, respectively; in the latter case the two flag-orbits can be represented by a pair of $1$-adjacent flags of $F_{n}/F$. Note, however, that unlike in the case of hereditary polytopes with regular facets, $\Lambda$ does not act trivially on the $(n-3)$-face $F/F_{-1}$. (In fact, each $\tau_{0,n-2}^H$ maps $\Omega$ to $\Omega^0$, the $0$-adjacent flag, so the restriction of $\Lambda$ to $F/F_{-1}$ is a group $C_2$.)  

Now let $G$ be an $(n-2)$-face of $\pp$ incident with $F$, and let $H$ and $H'$ denote the two facets of $\pp$ meeting at $G$. Then $\Phi:=\Omega \cup \{G,H,F_n\}$ is a flag of $\pp$ containing $F$. Note that $\{F,G,H,F_n\}$ and $\{F,G,H',F_n\}$ are $1$-adjacent flags of the $q$-gon $F_{n}/F$ which are contained in $\Phi$ or $\Phi^{n-1}$, respectively. Now let $\Psi$ be any flag of $\pp$ containing $F$. Then two possible scenarios can occur.

First suppose $q$ is odd. Then since $\Lambda$ acts flag-transitively on $F_{n}/F$, the flag $\Psi$ can be mapped by an element of $\Lambda$ to a flag $\Psi'$ containing $\{F,G,H,F_n\}$. Then $\Psi'\setminus \{F_n\}$ is a flag of the facet $H/F_{-1}$, and since $H/F_{-1}$ is chiral, it can be taken by an automorphism of $H/F_{-1}$ to either the flag $\Phi\setminus \{F_n\}$ of $H/F_{-1}$ or the $j$-adjacent flag $(\Phi\setminus \{F_n\})^j$, for any $j=0,\ldots,n-2$. But $\pp$ is hereditary, so the extension of this automorphism to $\pp$ then necessarily maps $\Psi'$ to either $\Phi$ or $\Phi^j$. On the other hand, the two flags $\Phi$ and $\Phi^j$ are not equivalent under $\Gamma(\pp)$, since otherwise the facets would be regular, not chiral. Thus $\Gamma(\pp)$ has two flag-orbits represented by any pair of $j$-adjacent flags, with $j=0,\ldots,n-2$. Hence $\pp$ is a $2$-orbit polytope, either of type $2_\emptyset$ and then $\pp$ is chiral, or of type $2_{\{n-1\}}$. (Note that our arguments do not require the above automorphisms to belong to $\Gamma_{F}(\pp)$; in fact, when $j=n-3$, and possibly when $j=n-2$ with $n\geq 5$, they will not lie $\Gamma_{F}(\pp)$.)

Next suppose $q$ is even. Now $\Lambda$ acts with two flag-orbits on $F_{n}/F$, so $\Psi$ can be mapped under $\Lambda$ to a flag $\Psi'$ which either contains $\{F,G,H,F_n\}$ or  $\{F,G,H',F_n\}$. In the former case, $\Psi'$ is as above equivalent to $\Phi$ or $\Phi^j$, for any $j=0,\ldots,n-2$, again under the extension of a suitable automorphism of the chiral facet $H/F_{-1}$ to $\pp$. In the latter case, $\Psi'$ is equivalent to $\Phi^{n-1}$ or $\Phi^{n-1,j}$, for any $j=0,\ldots,n-2$, now under the extended automorphism of the $(n-1)$-adjacent facet $H'/F_{-1}$ of $H/F_{-1}$ in $\pp$.  As before, $\Phi$ and $\Phi^j$ cannot be equivalent under $\Gamma(\pp)$, and neither can $\Phi^{n-1}$ and $\Phi^{n-1,j}$. Moreover, $\Phi$ is equivalent to $\Phi^{n-1}$ or $\Phi^{n-1,j}$ respectively, if and only if $\Phi^j$ is equivalent $\Phi^{n-1,j}$ or $\Phi^{n-1}$. Hence $\pp$ has two or four flag-orbits. If there are four flag-orbits, then these can be represented by $\Phi,\Phi^j,\Phi^{n-1},\Phi^{n-1,j}$, and we are done. Otherwise $\pp$ is a $2$-orbit polytope with its two flag-orbits represented by $\Phi$ and $\Phi^j$. In this case $\pp$ is either of type $2_\emptyset$ and then $\pp$ is chiral, or of type $2_{\{n-1\}}$; accordingly, $\Phi$ and $\Phi^{n-1}$ represent different, or the same, flag-orbits under $\Gamma(\pp)$. In either case we are done as well, and the proof is complete.
\qed\end{proof} 

Note that the proof of Theorem~\ref{chirfac} shows that the four flag-orbits of a chiral-facetted hereditary $4$-orbit $n$-polytope $\pp$ can be represented by the four flags $\Psi,\Psi^0,\Psi^{n-1},\Psi^{n-1,0}$, where $\Psi$ is any flag of $\pp$.

In rank $4$, many examples of chiral polytopes with chiral facets are known (see \cite{bjs,chp,SW2}). These are chiral-facetted hereditary polytopes of the first kind mentioned in Theorem~\ref{chirfac}. By contrast, it is not at all clear that chiral-facetted hereditary polytopes of the two other kinds actually exist (for any rank $n\geq 4$). In the remainder of this section we establish the existence of such examples. We show that there is a wealth of chiral-facetted hereditary $2$-orbit polytopes in the class $2_{\{n-1\}}$ for any $n\geq 4$; and that chiral-facetted hereditary $4$-orbit polytopes exist at least in rank $4$.

\subsection{Chiral-facetted hereditary $n$-polytopes in class $2_{\{n-1\}}$}
\label{chirfacher}

We begin by briefly reviewing the cube-like polytopes $2^\mathcal{K}$ originally due to Danzer (see \cite{dan,esext} and \cite[Section 8D]{arp}).  

Let $\mathcal{K}$ be a finite abstract $(n-1)$-polytope with vertex-set $V:=\{1,\ldots,v\}$ (say). Suppose $\mathcal{K}$ is {\em vertex-describable\/}, meaning that its faces are uniquely determined by their vertex-sets. Thus we may identify the faces of $\mathcal{K}$ with their vertex-sets, which are subsets of $V$. Then $2^\mathcal{K}$ is a (vertex-describable) abstract $n$-polytope with $2^v$ vertices, each with a vertex-figure isomorphic to $\mathcal{K}$. The vertex-set of $2^\mathcal{K}$ is 
\begin{equation}
\label{twov}
2^V := \bigotimes_{i=1}^{v} \{0,1\} ,
\end{equation}
the cartesian product of $v$ copies of $\{0,1\}$. When $j\geq 1$ we take as $j$-faces of $2^\mathcal{K}$, for any $(j-1)$-face $F$ of $\mathcal{K}$ and any $\varepsilon:=(\varepsilon_1,\ldots,\varepsilon_v)$ in $2^{V}$, the subsets $F(\varepsilon)$ of $2^V$ defined by
\begin{equation}
\label{fep}
F(\varepsilon) := \{(\eta_1,\ldots,\eta_{v})\in 2^V\! \mid \eta_{i} = \varepsilon_{i} \mbox{ if } i\not\in F\}, 
\end{equation}
or, abusing notation, by the cartesian product
\[ F(\varepsilon) := \left( \bigotimes_{i \in F} \{0,1\} \right) 
\times \left( \bigotimes_{i \not\in F} \{\varepsilon_i\} \right). \]
Then, if $F$, $F'$ are faces of $\mathcal{K}$ and $\varepsilon=(\varepsilon_1,\ldots,\varepsilon_v)$, $\varepsilon' =(\varepsilon_{1}',\ldots,\varepsilon_{v}')$ are elements in $2^{V}$, we have $F(\varepsilon) \subseteq F'(\varepsilon')$ in $2^\mathcal{K}$ if and only if $F \leq F'$ in $\mathcal{K}$ and $\varepsilon_i = \varepsilon_{i}'$ for each $i$ not in $F'$. The least face of $2^\mathcal{K}$ (of rank $-1$) is the empty set. Note that the vertices $\varepsilon$ of $2^\mathcal{K}$ arise here as singletons in the form $F(\varepsilon)=\{\varepsilon\}$ when $F=\emptyset$, the least face of $\mathcal{K}$. Notice that if $\mathcal{K}$ is the $(n-1)$-simplex, then $2^\mathcal{K}$ is the $n$-cube.

Each $j$-face of $2^\mathcal{K}$ is isomorphic to a $j$-polytope $2^\mathcal{F}$, where $\mathcal{F}$ is a $(j-1)$-face of $\mathcal{K}$. More precisely, if $F$ is a $(j-1)$-face of $\mathcal{K}$ and $\mathcal{F}:=F/F_{-1}$, then each $j$-face $F(\varepsilon)$ with $\varepsilon$ in $2^V$ is isomorphic to $2^\mathcal{F}$.

The automorphism group of $2^\mathcal{K}$ is given by 
\begin{equation}
\label{gr2k}
\Gamma(2^\mathcal{K}) \,\cong\, C_{2}\wr \Gamma(\mathcal{K}) \,\cong\, C_{2}^{\,v}\rtimes \Gamma(\mathcal{K}), 
\end{equation}
the wreath product of $C_2$ and~$\Gamma(\mathcal{K})$ defined by the natural action of $\Gamma(\mathcal{K})$ on the vertex-set of $\mathcal{K}$. In particular, $\Gamma(2^\mathcal{K})$ acts vertex-transitively on $2^\mathcal{K}$ and the vertex stabilizers are isomorphic to $\Gamma(\mathcal{K})$. Moreover, each automorphism of every vertex-figure of $2^\mathcal{K}$ extends to an automorphism of the entire polytope $2^\mathcal{K}$.

The following theorem summarizes properties of $2^\mathcal{K}$ that are relevant for our discussion of hereditary polytopes.

\begin{theorem}
\label{prop2k}
Let $\mathcal{K}$ be a finite abstract $(n-1)$-polytope with $v$ vertices, and let $\mathcal{K}$ be vertex-describable. Then $2^\mathcal{K}$ is a finite abstract $n$-polytope with the following properties.\\[.04in]
(a) If $\mathcal{K}$ is a $k$-orbit polytope for $k\geq 1$, then $2^\mathcal{K}$ is also a $k$-orbit polytope.\\[.02in]
(b) If $\mathcal{K}$ is regular, then $2^\mathcal{K}$ is regular.\\[.02in]
(c) If $\mathcal{K}$ is a $2$-orbit polytope in class $2_I$ for $I\subseteq \{0,\ldots,n-2\}$, then $2^\mathcal{K}$ is a $2$-orbit polytope in class $2_{J}$ for $J:=\{0\} \cup \{i+1\mid i\in I\}$. \\[.02in]
(d) If $\mathcal{K}$ is chiral, then $2^\mathcal{K}$ is a $2$-orbit polytope in class $2_{\{0\}}$.
\end{theorem}

\begin{proof}
For part (a), since $\Gamma(2^\mathcal{K})$ acts vertex-transitively on $2^\mathcal{K}$, every flag-orbit under $\Gamma(2^\mathcal{K})$ can be represented by a flag containing the vertex $o:=(0,\ldots,0)$ of $2^\mathcal{K}$. Moreover, since the vertex stabilizer of $o$ is isomorphic to $\Gamma(\mathcal{K})$, two flags containing $o$ are equivalent in $\Gamma(2^\mathcal{K})$ if and only if they are equivalent in $\Gamma(\mathcal{K})$. Thus the number of flag-orbits of $\mathcal{K}$ and $2^\mathcal{K}$ is the same. This proves part (a). For part (b), simply apply part (a) with $k=1$. 

For part (c), suppose $\mathcal{K}$ is a $2$-orbit polytope in class $2_I$. Then part (a) shows that $2^\mathcal{K}$ is also a $2$-orbit polytope. Choose a flag $\Psi:=\{F_0,F_1,\ldots,F_{n-2}\}$ of $\mathcal{K}$ and consider the corresponding flag 
$\Phi:=\{o,F_0(o),F_1(o),\ldots,F_{n-2}(o)\}$ of $2^\mathcal{K}$ which contains $o$ (we are suppressing the least face and the largest face). In $\mathcal{K}$, the $i$-adjacent flags $\Psi,\Psi^i$ of $\mathcal{K}$ lie in the same orbit under $\Gamma(\mathcal{K})$ if and only if $i\in I$. Relative to $2^\mathcal{K}$, the adjacency levels of $\mathcal{K}$ are shifted by $1$. Hence, if $j\geq 1$, then a pair of $j$-adjacent flags $\Phi,\Phi^j$ of $2^{\mathcal{K}}$ lie in the same orbit under $\Gamma(2^\mathcal{K})$ if and only if $j\in \{i+1\mid i\in I\}$. In addition, the $0$-adjacent flags $\Phi,\Phi^0$ of $2^{\mathcal{K}}$ always are equivalent under $\Gamma(2^\mathcal{K})$; in fact, the mapping on $2^V$ defined by 
\[ (\varepsilon_1,\varepsilon_2,\ldots,\varepsilon_v) \longrightarrow 
(\varepsilon_{1}+1,\varepsilon_2,\ldots,\varepsilon_v),\]
with addition mod $2$ in the first component, induces an automorphism of $2^\mathcal{K}$ taking $\Phi$ to $\Phi^0$. Thus, $\Phi$ and $\Phi^j$ are in the same flag-orbit of $2^\mathcal{K}$ if and only if $j\in J$. This proves part (c). For part (d), apply part (c) with $I=\emptyset$. 
\qed\end{proof}

Appealing to duality, the previous theorem now allows us to settle the existence of chiral-facetted hereditary $n$-polytopes in class $2_{\{n-1\}}$. Call an abstract polytope $\mathcal{Q}$ {\em facet-describable} if each face of $\mathcal{Q}$ is uniquely determined by the facets of $\mathcal{Q}$ that are incident with it. Thus, $\mathcal{Q}$ is facet-describable if and only if its dual $\mathcal{Q}^*$ is vertex-describable. Just like vertex-describability, facet-describability is a relatively mild assumption on a polytope. Any polytope that is a lattice, is both 
vertex-describable and facet-describable.

\begin{corollary}
\label{cor}
Let $\mathcal{Q}$ be a finite chiral $(n-1)$-polytope, and let $\mathcal{Q}$ be facet-describable. Then $(2^{\mathcal{Q}^{\,*}})^*$ is a chiral-facetted hereditary $2$-orbit $n$-polytope in class $2_{\{n-1\}}$ with facets isomorphic to $\mathcal{Q}$. Moreover, 
\[\Gamma((2^{\mathcal{Q}^{\,*}})^*) \,\cong\, C_{2}\wr \Gamma(\mathcal{Q}) \,\cong\, C_{2}^{\,f}\rtimes \Gamma(\mathcal{Q}),\] 
where $f$ is the number of facets of $\mathcal{Q}$.
\end{corollary}

\begin{proof}
The dual $\mathcal{Q}^*$ of $\mathcal{Q}$ is chiral and vertex-describable. By Theorem~\ref{prop2k}, the polytope $2^{\mathcal{Q}^{\,*}}$ has $2$-orbits and belongs to class $2_{\{0\}}$. Hence its dual, $(2^{\mathcal{Q}^{\,*}})^*$, is a $2$-orbit polytope in class $2_{\{n-1\}}$. Its facets are the duals of the vertex-figures of $2^{\mathcal{Q}^{\,*}}$. Thus the facets of $(2^{\mathcal{Q}^{\,*}})^*$ are isomorphic to $\mathcal{Q}$ and hence are chiral. Moreover, $(2^{\mathcal{Q}^{\,*}})^*$ is hereditary, since every automorphism of every vertex-figure of $2^{\mathcal{Q}^{\,*}}$  extends to an automorphism of the entire polytope $2^{\mathcal{Q}^{\,*}}$. The second part of the corollary follows from (\ref{gr2k}), bearing in mind that $f$ is just the number of vertices of $\mathcal{Q}^*$ and that dual polytopes have the same group.
\qed\end{proof}

Chiral polytopes are known to exist for every rank greater than or equal to $3$ (see Pellicer~\cite{pell}). We strongly suspect that most polytopes constructed in \cite{pell} are also facet-describable.  Corollary~\ref{cor} provides an $n$-polytope of the desired kind for every $n\geq 4$ for which there exists a finite chiral $(n-1)$-polytope which is facet-describable. For $n=4$ or $5$ there are many such examples. 

\subsection{Chiral-facetted hereditary polytopes with four-orbits}

In this section we describe a construction of ``alternating" polytopes which is inspired by the methods in Monson \& Schulte~\cite{monsch} and provides examples of chiral-facetted hereditary $4$-polytopes with four flag-orbits. 

\subsubsection{Halving of polyhedra}
\label{halv}

We begin by reviewing a construction of polyhedra which arises from the halving operation $\eta$ of \cite[Section 7B]{arp} described below; it can be considered as a special case of the construction given in~\ref{hpolyhedra}.

Let $\mathcal{K}$ be an equivelar map of type $\{4,q\}$ whose edge graph is bipartite. Then every edge circuit in $\mathcal{K}$ has even length.  Suppose that the vertices of $\mathcal{K}$ are colored {\em red\/} and {\em yellow\/} such that adjacent vertices have different colors. The vertex-figures at the red vertices of $\mathcal{K}$ (obtained by joining the vertices adjacent to a given red vertex in cyclic order) form the faces of a map of type $\{q,q\}$ (which is usually a polyhedron) on the same surface as the original map. Its vertices and ``face centers" are the yellow and red vertices of $\mathcal{K}$, respectively; its edges are the ``diagonals" in (square) faces of $\mathcal{K}$ that join yellow vertices. 

When the two colors are interchanged, we similarly obtain another map of type $\{q,q\}$, the dual of the first map, which a priori need not be isomorphic to the first map. However, if $\mathcal{K}$ admits an automorphism swapping the two color-classes of vertices, then these maps are isomorphic; this holds, for example, if the original polyhedron $\mathcal{K}$ is vertex-transitive. In our applications this will always be the case, and in such instances we denote the map by $\mathcal{K}^a$ (with the ``a" standing for ``alternate vertices").

We now impose symmetry conditions on $\mathcal{K}$. First let $\mathcal{K}$ be regular, and let $\Gamma(\mathcal{K}) =\langle\alpha_0,\alpha_1,\alpha_2\rangle$, where $\alpha_0,\alpha_1,\alpha_2$ are the distinguished generators. From the {\em halving operation\/} 
\begin{equation}
\label{eta}
\eta:\, (\alpha_0,\alpha_1,\alpha_2) \rightarrow 
(\alpha_0\alpha_1\alpha_0,\alpha_2,\alpha_1) =: (\beta_0,\beta_1,\beta_2) ,
\end{equation}
we then obtain the new generators $\beta_0,\beta_1,\beta_2$ for the automorphism group of a self-dual regular polyhedron $\mathcal{K}^\eta$ of type $\{q,q\}$, which is a subgroup of index $2$ in $\Gamma(\mathcal{K})$ (see \cite[Section 7B]{arp}); bear in mind here that the edge graph of $\mathcal{K}$ is bipartite and that $(\alpha_{0}\alpha_{1})^{4}=\varepsilon$. This polyhedron can be drawn as a map on the same surface as $\mathcal{K}$ by employing Wythoff's construction with generators $\beta_0,\beta_1,\beta_2$ and base vertex $z$ (say) of $\mathcal{K}$. Then it is easily seen that $\mathcal{K}^\eta$ is just the polyhedron $\mathcal{K}^a$ described earlier, realized here with $z$ as a yellow vertex of $\mathcal{K}$. 

Notice that replacing $\eta$ by 
\begin{equation}
\label{etaalt}
\eta^0:\, (\alpha_0,\alpha_1,\alpha_2) \rightarrow 
(\alpha_1,\alpha_2,\alpha_0\alpha_1\alpha_0) =: (\gamma_0,\gamma_1,\gamma_2) 
\end{equation}
results in another set of generators $\gamma_0,\gamma_1,\gamma_2$, which are conjugate under $\alpha_0$ to $\beta_0,\beta_1,\beta_2$.  When Wythoff's construction is applied with these new generators and base vertex $\alpha_{0}(z)$ adjacent to $z$, we similarly arrive at a regular polyhedron $\mathcal{K}^{\eta^{0}}$ on the same surface which is dually positioned to $\mathcal{K}^{\eta}$,  has its vertices at the red vertices of $\mathcal{K}$, and is isomorphic to $\mathcal{K}^a$. Note that the new generators $\gamma_0,\gamma_1,\gamma_2$ in (\ref{etaalt}) can be found from $\alpha_0,\alpha_1,\alpha_2$ in one of two equivalent ways:\ either as in $\eta^0$ by first applying $\eta$ and then conjugating the $\beta_j$'s by $\alpha_0$, or by first conjugating the $\alpha_j$'s by $\alpha_0$ and then applying $\eta$ to these new generators. 

If $\mathcal{K}$ is chiral, we can work with corresponding operations at the rotation group level, again denoted by $\eta$ and $\eta^0$. Suppose $\Gamma(\mathcal{K}) =\langle\sigma_1,\sigma_2\rangle$, where $\sigma_1,\sigma_2$ are the distinguished generators. Then the two operations
\begin{equation}
\begin{array}{rlll}
\label{etachir}
\eta: & (\sigma_1,\sigma_2) &\rightarrow &
(\sigma_{1}^2\sigma_{2},\sigma_{2}^{-1}) =: (\varphi_{1},\varphi_{2}) \\[.04in]
\eta^0:& (\sigma_1,\sigma_2) & \rightarrow &
(\sigma_{2},\sigma_{2}^{-1}\sigma_{1}^{2}) =: (\psi_{1},\psi_{2}) 
\end{array}
\end{equation}
give a pair of self-dual chiral maps of type $\{q,q\}$ each isomorphic to $\mathcal{K}^a$. These maps can be drawn on the same underlying surface by employing a variant of Wythoff's construction, now applied with the new generators of (\ref{etachir}) and with either $z$ or $\sigma_{1}(z)$ as base vertex. The two maps are again dually positioned relative to each other. The vertex $z$ of $\mathcal{K}$ is a vertex of $\mathcal{K}^\eta$, but not of $\mathcal{K}^{\eta^0}$. Hence, if $z$ is a yellow vertex of $\mathcal{K}$, then $\mathcal{K}^\eta$ uses only yellow vertices of $\mathcal{K}$ while $\mathcal{K}^{\eta^0}$ uses only red vertices of $\mathcal{K}$. In analogy to what we said about the operations in (\ref{eta}) and (\ref{etaalt}), the new generators $\psi_{1},\psi_{2}$ in $\eta^0$ of (\ref{etachir}) can be found from $\sigma_1,\sigma_2$ in one of two equivalent ways:\ either as in $\eta^0$ by first applying $\eta$ and then passing to the generators for the other enantiomorphic form of $\mathcal{K}^\eta$, or by first passing to the generators for the other enantiomorphic form of $\mathcal{K}$ and then applying $\eta$ to these new generators. 

\subsubsection{Alternating chiral-facetted $4$-polytopes} 

Following \cite{monsch}, an $n$-polytope is said to be {\em alternating\/} if it has facets of possibly two distinct combinatorial isomorphism types appearing in alternating fashion around faces of rank $n-3$. We allow the possibility that the two isomorphism types coincide, although we are less interested in this case. The cuboctahedron is an example of an alternating polyhedron in which triangles and squares alternate around a vertex. 

A more interesting example is the familiar semiregular tiling $\mathcal{T}$ of Euclidean $3$-space $\mathbb{E}^3$ by regular octahedra and tetrahedra illustrated in Figure~\ref{calT}, which is an alternating $4$-polytope in which octahedra and tetrahedra alternate around an edge (see \cite{cox2,monsch}). Its vertex-figures are (alternating) cuboctahedra. More generally it is true that the vertex-figures of an alternating $n$-polytopes are alternating $(n-1)$-polytopes. From now on, we restrict ourselves to polytopes of rank $3$ or $4$.

\hspace{3cm}
\begin{figure}
\psset{xunit=0.009cm,yunit=0.009cm,algebraic=true,dotstyle=o,dotsize=3pt 0,linewidth=0.8pt,arrowsize=3pt 2,arrowinset=0.25}
\begin{pspicture*}(-14.93,-16.81)(1019.62,839.08)
\pspolygon[linestyle=none,fillstyle=solid,fillcolor=yellow,opacity=0.55](64.7,537.1)(213.47,189.38)(517.94,71.74)
\pspolygon[linestyle=none,fillstyle=solid,fillcolor=ffffzz,opacity=0.5](671.91,639.17)(64.7,537.1)(517.94,71.74)
\psline[linestyle=dotted,linecolor=yellow](64.7,537.1)(213.47,189.38)
\psline[linestyle=dotted,linecolor=yellow](213.47,189.38)(517.94,71.74)
\psline[linestyle=dotted,linecolor=yellow](517.94,71.74)(64.7,537.1)
\psline[linecolor=ffffzz](671.91,639.17)(64.7,537.1)
\psline[linecolor=ffffzz](64.7,537.1)(517.94,71.74)
\psline[linecolor=ffffzz](517.94,71.74)(671.91,639.17)
\psline(215.2,644.36)(671.91,639.17)
\psline(671.91,639.17)(671.91,187.65)
\psline(517.94,71.74)(671.91,187.65)
\psline[linewidth=1.6pt,linecolor=red](365.71,364.1)(59.51,78.66)
\psline[linewidth=1.6pt,linecolor=ffttqq](365.71,364.1)(287.86,76.93)
\psline[linewidth=1.6pt,linecolor=ffttqq](365.71,364.1)(517.94,71.74)
\psline[linewidth=1.6pt,linecolor=ffttqq](517.94,71.74)(59.51,78.66)
\psline[linewidth=1.6pt,linestyle=dashed,dash=2pt 2pt,linecolor=ffttqq](367.79,129.76)(59.51,78.66)
\psline[linestyle=dashed,dash=2pt 2pt,linecolor=ffttqq](671.91,187.65)(367.79,129.76)
\psline[linewidth=1.6pt,linestyle=dashed,dash=2pt 2pt,linecolor=ffttqq](367.79,129.76)(365.71,364.1)
\psline[linestyle=dashed,dash=2pt 2pt,linecolor=ffttqq](671.91,639.17)(365.71,364.1)
\psline[linestyle=dashed,dash=2pt 2pt,linecolor=ffttqq](367.44,748.15)(365.71,364.1)
\psline(64.7,537.1)(59.51,78.66)
\psline(64.7,537.1)(215.2,644.36)
\psline[linestyle=dotted](215.2,644.36)(213.47,189.38)
\psline(64.7,537.1)(519.67,533.64)
\psline(519.67,533.64)(671.91,639.17)
\psline(519.67,533.64)(517.94,71.74)
\psline[linestyle=dotted](59.51,78.66)(213.47,189.38)
\psline[linestyle=dotted](671.91,187.65)(213.47,189.38)
\psline[linewidth=1.6pt,linestyle=dashed,dash=2pt 2pt,linecolor=ffttqq](367.79,129.76)(287.86,76.93)
\psline(64.7,537.1)(517.94,71.74)
\psline(213.47,189.38)(517.94,71.74)
\psline(213.47,189.38)(64.7,537.1)
\psline[linestyle=dashed,dash=2pt 2pt,linecolor=ffttqq](265.37,268.96)(367.79,129.76)
\psline[linestyle=dashed,dash=2pt 2pt,linecolor=ffttqq](265.37,268.96)(517.94,71.74)
\psline[linestyle=dashed,dash=2pt 2pt,linecolor=ffttqq](671.91,639.17)(789.55,737.77)
\psline[linestyle=dashed,dash=2pt 2pt,linecolor=ffttqq](671.91,187.65)(798.19,213.6)
\psline[linestyle=dashed,dash=2pt 2pt,linecolor=ffttqq](697.86,71.74)(517.94,71.74)
\parametricplot[linecolor=qqcctt]{1.1086592906632604}{1.7}{1*85.36*cos(t)+0*85.36*sin(t)+531.78|0*85.36*cos(t)+1*85.36*sin(t)+34.55}
\parametricplot[linecolor=qqcctt]{3.7507913937560695}{6.024414415215889}{1*65.73*cos(t)+0*65.73*sin(t)+363.11|0*65.73*cos(t)+1*65.73*sin(t)+615.81}
\parametricplot[linecolor=qqcctt]{3.259325228393887}{4.379769288810777}{1*121.81*cos(t)+0*121.81*sin(t)+710.56|0*121.81*cos(t)+1*121.81*sin(t)+654.38}
\rput[tl](332.93,387){$F_3$}
\rput[tl](39.61,570){$C$}
\rput[tl](203.81,175){$B$}
\rput[tl](46.07,65){$F_0$}
\rput[tl](269.75,65){$F_{1}$}
\rput[tl](362.83,120){$F_2$}
\rput[tl](486.44,65){$F_0' = A$}
\rput[tl](710.62,90){$\sigma_3$}
\rput[tl](803.71,235){$\sigma_2 \sigma_3$}
\rput[tl](789.49,765){$\sigma_2$}
\rput[tl](356.37,780){$\sigma_1$}
\rput[tl](247.77,300){$E$}
\psline[linecolor=qqcctt](630,562)(650,562)
\psline[linecolor=qqcctt](630,562)(645,542)
\rput[tl](339.56,559){$\mathbf{\qqcctt{<}}$}
\rput[tl](537.37,127){$\mathbf{\qqcctt{>}}$}
\rput[tl](400,765){4-fold rotation}
\rput[tl](764.92,711.09){4-fold rotation}
\rput[tl](783.02,199.1){half-turn}
\rput[tl](662.78,60.76){4-fold rotation}
\rput[tl](655.03,675){$D$}
\psdots[dotsize=5pt 0,dotstyle=*,linecolor=yellow](64.7,537.1)
\psdots[dotsize=5pt 0,dotstyle=*,linecolor=yellow](213.47,189.38)
\rput[bl](218.03,201.69){}
\psdots[dotsize=5pt 0,dotstyle=*,linecolor=fffftt](517.94,71.74)
\rput[bl](523.15,84.03){}
\psdots[dotstyle=*,linecolor=ffttqq](671.91,639.17)
\psdots[dotstyle=*,linecolor=ffttqq](265.37,268.96)
\psdots[dotstyle=*,linecolor=ffttqq](59.51,78.66)
\psdots[dotstyle=*,linecolor=ffttqq](287.86,76.93)
\psdots[dotstyle=*,linecolor=ffttqq](367.79,129.76)
\psdots[dotstyle=*,linecolor=ffttqq](365.71,364.1)
\end{pspicture*}
\caption{A patch of the semiregular tiling $\mathcal{T}$ derived from the cubical tiling $\mathcal{C}$. Shown are a tetrahedal tile with vertices $A,B,C,D$, and one eighths of an octahedral tile (vertices $A,B,C$) centered at the base vertex $F_{0}=z$. The axes of the three generating rotations $\sigma_1,\sigma_2,\sigma_3$ for the rotation subgroup of $\mathcal{C}$ are indicated, as is the fundamental tetrahedron for this subgroup with vertices $F_{0},F_{0}',F_{2},F_{3}$. The plane through $A,B,C$ dissects this fundamental tetrahedron into two smaller tetrahedra, each becoming a fundamental tetrahedron for the full symmetry group of a tile, namely the tetrahedron with vertices $F_{0},F_{0}',F_{2},E$ for the octahedral tile and the tetrahedron with vertices $F_{0}',F_{2},E,F_{3}$ for the tetrahedral tile.  \label{calT}
}
\end{figure}
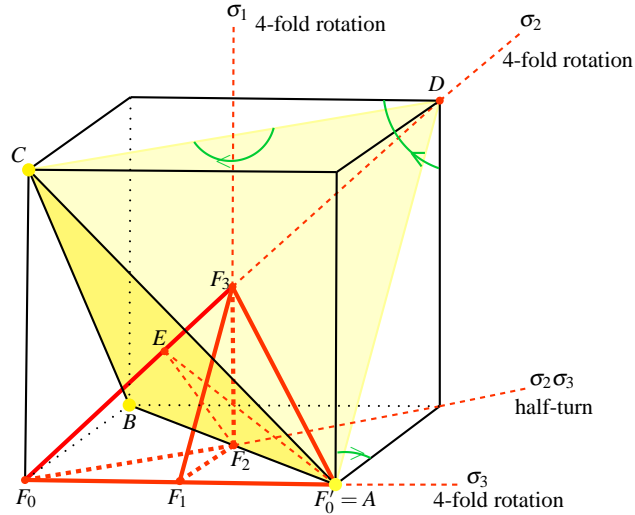

The relationship of the semiregular tiling $\mathcal{T}$ with the (regular) cubical tiling $\mathcal{C}:=\{4,3,4\}$ in $\mathbb{E}^3$ will serve as the blueprint for our construction. As the edge graph of $\mathcal{C}$ is bipartite, we can color the vertices {\em red\/} or {\em yellow} such that adjacent vertices receive different colors. Then the octahedral tiles of $\mathcal{T}$ can be viewed as the vertex-figures of $\mathcal{C}$ at the red vertices, each spanned by the yellow vertices adjacent to the corresponding red vertex. The complement in $\mathbb{E}^3$ of the union of all these octahedral tiles gives rise to the family of tetrahedral tiles of $\mathcal{T}$, each inscribed in a cube of $\mathcal{C}$; each cube contributes exactly one tetrahedral tile, such that the tetrahedral tiles in adjacent cubes share a common edge. 

Now let $\mathcal{P}$ be any finite $4$-polytope, let $\mathcal{K}$ be a vertex-transitive polyhedron of type $\{4,q\}$,
and let $\mathcal{L}$ be a polyhedron of type $\{q,r\}$. Suppose that all facets of $\mathcal{P}$ are isomorphic to $\mathcal{K}$, and that all vertex-figures are isomorphic to $\mathcal{L}$. Thus $\mathcal{P}$ is equivelar of type $\{4,q,r\}$.

Further, suppose the edge graph of $\mathcal{P}$ is bipartite, with vertices colored red or yellow such that adjacent vertices have different colors. Let $R$ and $Y$, respectively, denote the sets of red or yellow vertices of $\mathcal{P}$. Then every edge circuit in $\mathcal{P}$ has even length, and the edge graph of $\mathcal{K}$ is also bipartite. It is convenient to require two additional ``lattice-like" conditions to hold. First, both $\mathcal{P}$ and $\mathcal{L}$ should be vertex-describable, so that we may identify faces with their vertex-sets; then, as a facet of a vertex-describable polytope, $\mathcal{K}$ must also be vertex-describable. Second, any two opposite vertices of a $2$-face of $\mathcal{P}$ should not be opposite vertices of another $2$-face of $\mathcal{P}$. Later we impose strong symmetry conditions on $\mathcal{K}$, $\mathcal{L}$ and $\mathcal{P}$, but for now we work in the present generality.

We now derive from $\mathcal{P}$ a new $4$-polytope $\mathcal{P}^a$, where ``$a$" indicates ``alternating". The vertex-set of $\mathcal{P}^a$ is the set $Y$ of yellow vertices of $\mathcal{P}$. Our description of the faces of $\mathcal{P}^a$ is in terms of their vertex-sets, that is, subsets of $Y$. In particular, the edges of $\mathcal{P}^a$ are the diagonals of the (square) $2$-faces of $\mathcal{P}$ that connect yellow vertices; more precisely, $\{v,w\}$ is a $1$-face of $\mathcal{P}^a$ if and only if $v,w\in Y$ and $v,w$ are opposite vertices in a $2$-face of $\mathcal{P}$. Then, by our assumption on the $2$-faces of $\mathcal{P}$, any two vertices of $\mathcal{P}^a$ are joined by at most one edge.
 
The $2$-faces of $\mathcal{P}^a$ are the vertex-figures, within the facets of $\mathcal{P}$, at the red vertices of these facets; more precisely, $\{v_{1},\ldots,v_{q'}\}$ is a $2$-face of $\mathcal{P}^a$ if and only if there exists a facet $F$ of $\mathcal{P}$ with a red vertex $v$ such that $\{v_{1},\ldots,v_{q'}\}$ is the set of (yellow) vertices, labeled in cyclic order, of the vertex-figure at $v$ in $F$. Clearly, the $2$-faces of $\mathcal{P}^a$ must be $q$-gons, that is, $q'=q$ in each case. Alternatively, we can describe the $2$-faces of $\mathcal{P}^a$ as the $2$-faces of the vertex-figures at red vertices in $\mathcal{P}$. 

The facets of $\mathcal{P}^a$ are of two kinds and correspond to either a halved facet or the vertex figure at a red vertex of $\mathcal{P}$. Each facet $F$ of $\mathcal{P}$ gives rise to a facet $F^a$ of $\mathcal{P}^a$, of the {\em first kind\/}, obtained (as in Section~\ref{halv}) as the polyhedron whose $2$-faces are the vertex-figures of $F$ at the red vertices; when $F$ is viewed as a map of type $\{4,q\}$ on a surface, $F^a$ is a map of type $\{q,q\}$ that can be drawn on the same surface. Note here that, by the vertex-transitivity of $\mathcal{K}$, the combinatorial structure of $F^a$ does not depend on which class of vertices in the bipartition of the vertex-set of $F$ is used as the vertex-set for $F^a$ (the two maps arising from the two possible choices of vertex-sets are related by duality, but they are isomorphic since $\mathcal{K}$ is vertex-transitive). Thus the facets $F^a$ of the first kind are mutually isomorphic, each to the map $\mathcal{K}^a$ of Section~\ref{halv}. The facets of $\mathcal{P}^a$ of the {\em second kind\/} are the vertex-figures, $\mathcal{P}\!/\!v$, of $\mathcal{P}$ at the red vertices,~$v$. 

For example, if $\mathcal{P}$ is the cubical tessellation $\mathcal{C}$ described earlier, then the facets of the first kind are tetrahedra $F^a=\{3,3\}$ inscribed in cubes $F$ of $\mathcal{C}$, and the facets of the second kind are the octahedral vertex-figures $\mathcal{C}\!/\!v = \{3,4\}$ of $\mathcal{C}$ at the red vertices. Thus, combinatorially, $\mathcal{P}^a = \mathcal{T}$, the semiregular tiling of $\mathbb{E}^3$ by tetrahedra and octahedra.

Incidence of faces in $\mathcal{P}^a$ is defined by inclusion of vertex-sets; that is, two faces of $\mathcal{P}^a$ are incident if and only if their vertex-sets (as subsets of the vertex-set of $\mathcal{P}$) are related by inclusion. Note that two facets of $\mathcal{P}^a$ can only be adjacent (share a $2$-face) if they are of different kinds, and that a facet $F^a$ of the first kind is adjacent to a facet $\mathcal{P}\!/\!v$ of the second kind if and only if $v$ is a vertex of $F$. Each edge of $\mathcal{P}^a$ is surrounded by four facets of $\mathcal{P}^a$, occurring in alternating fashion; more explicitly, if $\{v,w\}$ is an edge of $\mathcal{P}^a$ given by the diagonal of a $2$-face $G$ of $\mathcal{P}$, then these four facets are $F^a$, $\mathcal{P}\!/\!u$, $(F')^a$ and $\mathcal{P}\!/\!u'$, in this order, where $F$ and $F'$ are the two facets of $\mathcal{P}$ meeting at $G$, and $u,u'$ are the two vertices of $G$ distinct from $v$ and $w$. Thus $\mathcal{P}^a$ is alternating. 

The vertex-set of the vertex-figure $\mathcal{P}^a\!/\!v$ of $\mathcal{P}^a$ at a vertex $v$ (a yellow vertex of $\mathcal{P}$) consists of the vertices $w$ of $\mathcal{P}^a$ such that $\{v,w\}$ is an edge of $\mathcal{P}^a$. Combinatorially, $\mathcal{P}^a\!/\!v$ is the medial $Me(\mathcal{L})$ of the vertex-figure $\mathcal{L}$ of $\mathcal{P}$. To see this, in the above, replace the vertex $w$ of the edge $\{v,w\}$ by the ``midpoint" of that edge (this is equivalent to the ``center" of the respective $2$-face of $\mathcal{P}$ that determines that edge), and impose on this new vertex-set the same combinatorial structure as on the original vertex-set of $\mathcal{P}^a\!/\!v$. In the example of the semiregular tiling $\mathcal{T}$ of $\mathbb{E}^3$ the vertex-figures are cuboctahedra, occurring as medials of the octahedral vertex-figures of the cubical tiling $\mathcal{C}$ at yellow vertices.

Notice that the new polytope $\mathcal{P}^a$ has the same number of flags as the original polytope $\mathcal{P}$. In fact, the number of vertices of $\mathcal{P}^a$ is half that of $\mathcal{P}$, while the number of flags of the vertex-figures $Me(\mathcal{L})$ of $\mathcal{P}^a$ is twice that of the vertex-figures $\mathcal{L}$ or $\mathcal{P}$. Bear in mind our assumption that $\mathcal{P}$ is finite.

We now investigate the combinatorial symmetries of $\mathcal{P}^a$. First observe that $\mathcal{P}^a$ inherits all automorphisms of $\mathcal{P}$ that preserve colors of vertices. Observe here that, since the edge graph of $\mathcal{P}$ is bipartite and connected, an automorphism $\gamma$ of $\mathcal{P}$ maps the full set of yellow vertices $Y$ to itself if and only if $\gamma$ maps any yellow vertex to a yellow vertex. Let $\Gamma^{c}(\mathcal{P})$ denote the subgroup of $\Gamma(\mathcal{P})$ mapping $Y$ (and thus $R$) to itself. Clearly, $\Gamma^{c}(\mathcal{P})$ has index $1$ or $2$ in $\Gamma(\mathcal{P})$. Then it is immediately clear that $\Gamma^{c}(\mathcal{P})$ is a subgroup of $\Gamma(\mathcal{P}^a)$. In fact, the combinatorics of $\mathcal{P}^a$ is entirely derived from $Y$ and has been described in a $Y$-invariant fashion. 

With an eye on the hereditary property, we remark further that the vertex stabilizer $\Gamma_v(\mathcal{P})$ of a red vertex $v$ in $\Gamma(\mathcal{P})$ becomes a subgroup of the automorphism group of the corresponding facet $\mathcal{P}\!/\!v$ of $\mathcal{P}^a$. Similarly, for any facet $F$ of $\mathcal{P}$, the subgroup of color preserving automorphisms of $\Gamma(\mathcal{P})$, which is given by $\Gamma^{c}(\mathcal{P})\cap \Gamma(F/F_{-1})$,  becomes a subgroup of the automorphism group of the corresponding facet $F^a$ of $\mathcal{P}^a$.

Our remarks about $\Gamma^{c}(\mathcal{P})$ have immediate implications for the number of flag-orbits of $\mathcal{P}^a$. 

In particular, if $\mathcal{P}$ is regular, then $\Gamma^{c}(\mathcal{P})$ must have index $2$ as a subgroup of $\Gamma(\mathcal{P})$, and thus index $1$ or $2$ as a subgroup of $\Gamma(\mathcal{P}^a)$. To see this, note that the order of $\Gamma^{c}(\mathcal{P})$ is exactly half the number of flags of $\mathcal{P}$, and thus of $\mathcal{P}^a$. Hence $\mathcal{P}^a$ is regular or a $2$-orbit polytope in class $2_{\{0,1,2\}}$. In either case, $\mathcal{P}^a$ is hereditary (and regular-facetted). 

Similarly, if $\mathcal{P}$ is chiral, then $\Gamma^{c}(\mathcal{P})$ must have index $2$ as a subgroup of $\Gamma(\mathcal{P})$, and thus index $1$, $2$ or $4$ as a subgroup of $\Gamma(\mathcal{P}^a)$. Now the order of $\Gamma^{c}(\mathcal{P})$ is exactly a quarter of the number of flags of $\mathcal{P}$, and thus of $\mathcal{P}^a$. Now suppose $\mathcal{P}^a$ is hereditary. We show that then the facets and vertex-figures of $\mathcal{P}$ must be all regular or all chiral.

In fact, if the facets of the original polytope $\mathcal{P}$ are regular, each facet $F^a$ of $\mathcal{P}^a$ of the first  kind must also be regular and its full automorphism group must be a subgroup of $\Gamma(\mathcal{P}^a)$ (see Section~\ref{halv}); now since the combinatorial reflection symmetry in $F^a$ that takes a flag of $F^a$ to its $0$-adjacent flag also gives a similar such reflection symmetry in the adjacent facet $\mathcal{P}\!/\!v$ (say) of $\mathcal{P}^a$ meeting $F^a$ in the $2$-face of the flag, it follows that the vertex-figures of $\mathcal{P}$ must actually also be regular since they already have (at least) maximal symmetry by rotation. Similarly, if the vertex-figures of the original polytope $\mathcal{P}$ are regular, then the hereditary property of $\mathcal{P}^a$ implies that the full automorphism group $\Gamma(\mathcal{P}\!/\!v)$ of a facet $\mathcal{P}\!/\!v$ of $\mathcal{P}^a$ is a subgroup of $\Gamma(\mathcal{P}^a)$ containing a combinatorial reflection symmetry of $\mathcal{P}\!/\!v$ that takes a flag of $\mathcal{P}\!/\!v$ to its $0$-adjacent flag; as above, this reflection symmetry must induce a similar reflection symmetry in an adjacent facet $F^{a}$ (say) of $\mathcal{P}^a$ and hence force this facet to be regular, since it already has (at least) maximal symmetry by rotation. Thus, if the original polytope $\mathcal{P}$ is chiral, then $\mathcal{P}^a$ can be hereditary only if the facets and vertex-figures of $\mathcal{P}^a$ are all regular or all chiral. 

Conversely, if the facets and vertex-figures of a chiral polytope $\mathcal{P}$ are all regular or all chiral, then the new polytope $\mathcal{P}^a$ is hereditary, since each facet of either kind has all its automorphisms extended to the entire polytope $\mathcal{P}^a$. In particular, if the facets and vertex-figures of $\mathcal{P}$ are all regular, then $\mathcal{P}^a$ is regular-facetted and is either itself regular or a $2$-orbit polytope of type $2_{\{0,1,2\}}$. Otherwise, $\mathcal{P}^a$ is chiral-facetted and has $1$, $2$ or $4$ flag-orbits.

Now suppose $\mathcal{P}$ and all its facets and vertex-figures are chiral. Then recall from Section~\ref{flobfc} that the flag-orbits of the corresponding hereditary polytope $\mathcal{P}^a$ can be represented by one, two, or four flags from among $\Psi$, $\Psi^0$, $\Psi^{3}$, $\Psi^{3,0}$, where $\Psi$ is any flag of $\mathcal{P}^a$. First note that a pair of $0$-adjacent flags of $\mathcal{P}^a$ cannot possibly be equivalent under $\Gamma(\mathcal{P}^a)$, since otherwise the facet of $\mathcal{P}^a$ common to both flags would have to be regular, not chiral. Thus $\Psi,\Psi^0$ (resp. $\Psi^{3},\Psi^{3,0}$) are not equivalent under $\Gamma(\mathcal{P}^a)$, and $\mathcal{P}^a$ has $2$ or $4$ flag-orbits. Similarly, if the two kinds of facets of $\mathcal{P}^a$ are distinct (that is, non-isomorphic), then a pair of $3$-adjacent flags of $\mathcal{P}^a$ cannot possibly be equivalent either, since any automorphism of $\mathcal{P}^a$ taking a flag to its $3$-adjacent flag would provide an isomorphism between the two facets contained in these flags. Thus $\Psi,\Psi^{3}$ (resp. $\Psi^0,\Psi^{3,0}$) are non-equivalent and $\mathcal{P}^a$ must have $4$ flag-orbits. Note that the non-isomorphism condition on the two kinds of facets of $\mathcal{P}^a$ holds, for example, if their numbers of flags are distinct, that is, if the number of flags of $\mathcal{K}$ is not exactly twice that of $\mathcal{L}$.

Our main findings are summarized in the following theorem.

\begin{theorem}
\label{alter}
Let $\mathcal{P}$ be a finite regular or chiral $4$-polytope of type $\{\mathcal{K},\mathcal{L}\}$, where $\mathcal{K}$ and $\mathcal{L}$ are polyhedra of type $\{4,q\}$ and $\{q,r\}$, respectively. Suppose that the edge graph of $\mathcal{P}$ is bipartite, that $\mathcal{P}$ and $\mathcal{L}$ are vertex-describable, and that any two opposite vertices of a $2$-face of $\mathcal{P}$ are not opposite vertices of another $2$-face of $\mathcal{P}$. Then $\mathcal{P}^a$ is a finite alternating hereditary $4$-polytope with facets isomorphic to $\mathcal{L}$ or $\mathcal{K}^a$, and with vertex-figures isomorphic to the medial $Me(\mathcal{L})$ of $\mathcal{L}$. Every edge of $\mathcal{P}^a$ is surrounded by four facets, two of each kind occurring in an alternating fashion.  Moreover, $\mathcal{P}^a$ has the following hereditary properties.\\[.01in]
\indent$\!$ (a) If $\mathcal{K}$ and $\mathcal{L}$ are regular, then $\mathcal{P}^a$ is a regular-facetted hereditary polytope 
\indent\indent$\,$ 
and is either itself regular or a $2$-orbit polytope of type $2_{\{0,1,2\}}$.\\[.015in]
\indent$\!$ (b) If $\mathcal{K}$ and $\mathcal{L}$ are chiral, then $\mathcal{P}^a$ is a chiral-facetted hereditary polytope 
\indent\indent$\;$ 
with~$2$ or $4$ flag-orbits. If $\mathcal{L}$ and $\mathcal{K}^a$ are not isomorphic (for example, this 
\indent\indent$\;$
holds when $|\Gamma(\mathcal{L})|\neq |\Gamma(\mathcal{K})|/2$), then $\mathcal{P}^a$ has $4$ flag-orbits.\\[.015in]
In either case (a) or (b), the group of all color preserving automorphisms $\Gamma^{c}(\mathcal{P})$ of $\mathcal{P}$ is a subgroup of $\Gamma(\mathcal{P}^a)$ of index $1$ or $2$, with the same or twice the number of flag-orbits as $\Gamma(\mathcal{P}^a)$.
\end{theorem}

The construction summarized in the previous theorem is a rich source for interesting examples of chiral-facetted hereditary $4$-polytopes with $4$ flag-orbits. To begin with, suppose $\mathcal{P}$ is a finite chiral $4$-polytope of type $\{\mathcal{K},\mathcal{L}\}$ such that $\mathcal{K},\mathcal{L}$ are chiral and $\mathcal{K}^a,\mathcal{L}$ are non-isomorphic. There is a wealth of polytopes of this kind. Now, if the edge graph of $\mathcal{P}$ is bipartite, 
$\mathcal{P}$ and $\mathcal{L}$ are vertex-describable, and any two opposite vertices of a $2$-face of $\mathcal{P}$ are not opposite vertices of another $2$-face of $\mathcal{P}$, then Theorem~\ref{alter} applies and yields a chiral-facetted alternating $4$-polytope $\mathcal{P}^a$ which is hereditary and has $4$ flag-orbits. Thus we need to assure that these three conditions hold; the requirement of a bipartite edge graph seems to be the most severe condition among the three. In our examples described below we verified these conditions with \textsc{Magma}.
 
For example, starting with the universal $4$-polytope $\mathcal{P}=\{\{4,4\}_{1,3},\{4,4\}_{1,3}\}$, which has 50 vertices, 50 facets, and an automorphism group of size 2000, our construction yields a hereditary $4$-orbit polytope $\mathcal{P}^a$ which has two kinds of chiral facets, namely $\{4,4\}_{1,3}$ and $\{4,4\}_{1,2}$.  It can be seen, for example using \textsc{Magma}, that the universal $4$-polytope with the same facets but the enantiomorphic vertex-figures fails the conditions of Theorem~\ref{alter}, in that there exist two opposite vertices of a $2$-face which are opposite vertices of another $2$-face of that polytope.

\section{Extensions of hereditary polytopes}
\label{exther}

In this section we briefly discuss extension problems for hereditary polytopes. We begin with a generalization of the notion of a hereditary polytope.

Let $1 \leq j \leq n-1$. An $n$-polytope $\pp$ is said to be {\em $j$-face hereditary\/} if for each $j$-face $F$ of $\pp$, the automorphism group $\Gamma(F/F_{-1})$ of the section $F/F_{-1}$ is a subgroup of $\Gamma(\pp)$ (and hence of $\Gamma_{F}(\pp)$). Thus $\pp$ is $j$-face hereditary if every automorphism of a $j$-face $F$ extends to an automorphism of $\pp$.  Note that a hereditary polytope is $(n-1)$-face hereditary, or {\em facet hereditary}. 

A $j$-face hereditary polytope is {\em strongly $j$-face hereditary\/} if for each $j$-face $F$ of $\pp$, the automorphism group $\Gamma(F/F_{-1})$ is a subgroup of $\Gamma(\pp)$ acting trivially on the \textit{co-face} $F_{n}/F$; then $\Gamma(F/F_{-1})$ is the stabilizer of a flag of $F_{n}/F$ in $\Gamma_{F}(\pp)$. Thus, for a strongly $j$-face hereditary polytope, every automorphism of a $j$-face $F$ extends to a particularly well-behaved automorphism of $\pp$, namely one which fixes every face of $\pp$ in the co-face of $F$ in $\pp$. 

The (vertex) truncated tetrahedron is a $1$-face (or edge-) hereditary polyhedron which is not $2$-face hereditary. The perpendicular bisectors of its edges are mirrors of reflection, but no geometric symmetry or combinatorial automorphism can rotate the vertices of a single face by one step. This example is a $3$-orbit polyhedron.

Note that every $2$-orbit $n$-polytope in a class $2_I$ with $\{0,1,\ldots,j-1\}\subseteq I$ is a strongly $j$-face hereditary polytope with regular $j$-faces. This follows directly from the definition of the class $2_I$. For example, a $2$-orbit polytope of rank $4$ and type $2_{\{0,1\}}$ is $2$-face hereditary; it may also be $3$-face hereditary, but not a priori so.
 
Now the basic question arises whether or not each hereditary $n$-polytope occurs as a facet of an $(n-1)$-face hereditary $(n+1)$-polytope; or more generally, whether or not each $j$-face hereditary $n$-polytope occurs as a $j$-face of a $k$-face hereditary $(n+1)$-polytope, for any $j\leq k\leq n$.

In this context the following result is of interest. 

\begin{theorem}
\label{twoP}
Let $\mathcal{K}$ be a finite $j$-face hereditary $n$-polytope for some $j=1,\ldots,n-1$, and let $\mathcal{K}$ be vertex-describable. Then $\mathcal{K}$ is the vertex-figure of a vertex-transitive finite $(j+1)$-face hereditary $(n+1)$-polytope.  
\end{theorem}

\begin{proof}
We employ the $2^\mathcal{K}$ construction described in Section~\ref{chirfacher}. Since $\mathcal{K}$ is a vertex-describable finite $n$-polytope, $2^\mathcal{K}$ is a vertex-transitive finite $(n+1)$-polytope all of whose vertex-figures are isomorphic to $\mathcal{K}$. Every $(j+1)$-face of $2^\mathcal{K}$ is isomorphic to a $(j+1)$-polytope $2^\mathcal{F}$, where $\mathcal{F}:=F/F_{-1}$ is the $j$-polytope given by a $j$-face $F$ as $\mathcal{K}$.  Moreover, $\Gamma(2^\mathcal{K}) \cong C_{2}^{v}\rtimes \Gamma(\mathcal{K})$, where $v$ is the number of vertices of $\mathcal{K}$; similarly, $\Gamma(2^\mathcal{F}) \cong C_{2}^{v(\mathcal{F})}\rtimes \Gamma(\mathcal{F})$, where $v(\mathcal{F})$ is the number of vertices of $\mathcal{F}$ (that is, the number of vertices of $F$ in $\mathcal{K}$).
In particular, the automorphism group of any $(j+1)$-face $2^\mathcal{F}$ of $2^\mathcal{K}$ is a subgroup of $\Gamma(2^\mathcal{K})$ if $\mathcal{K}$ is $j$-face hereditary, since then $\Gamma(\mathcal{F})$ is a subgroup of $\Gamma(\mathcal{K})$. Thus $2^\mathcal{K}$ is a $(j+1)$-face hereditary $(n+1)$-polytope if $\mathcal{K}$ is a $j$-face hereditary $n$-polytope.
\qed\end{proof}

When $j=n-1$ we have the following immediate consequence.

\begin{corollary}
\label{jface2k}
Each finite vertex-describable hereditary $n$-polytope is the vertex-figure of a vertex-transitive finite hereditary $(n+1)$-polytope.
\end{corollary}

Theorem~\ref{twoP} and its proof are good sources for interesting examples of hereditary polytopes. For instance, if $\mathcal{K}$ is the truncated tetrahedron, which is $1$-face hereditary but not $2$-face hereditary, then $2^\mathcal{K}$ is a $2$-face hereditary $4$-polytope which is not $3$-face hereditary. In fact, the facets of $2^\mathcal{K}$ are of two kinds, $3$-cubes $\{4,3\}=2^{\{3\}}$ and orientable regular maps $\{4,6\,|\,4,4\}=2^{\{6\}}$ of genus $9$ (see \cite[p. 261]{arp}); however, not all automorphisms of facets of the latter kind extend to automorphisms of $2^\mathcal{K}$ (otherwise $\mathcal{K}$ would have to be $2$-hereditary). Similar examples of arbitrary higher ranks can be constructed by iterating the $2^\mathcal{K}$ construction. For example, when $\mathcal{K}$ is the truncated tetrahedron, $2^{2^\mathcal{K}}$ is a $3$-face, but not $4$-face, hereditary $5$-polytope.

Note that a further generalization of hereditary polytopes employs sections rather than faces. For $0 \leq i < j \leq n-1$, an $n$-polytope $\pp$ is said to be {\em $(i,j)$-section hereditary\/} (resp. {\em strongly $(i,j)$-section hereditary\/}) if for each section $G/F$, with $F$ an $i$-face and $G$ a $j$-face with $F<G$, the group $\Gamma(G/F)$ of $G/F$ is a subgroup of $\Gamma(\pp)$ (resp. fixing, in addition, each face in both $F/F_{-1}$ and $F_{n}/G$). 
\medskip

\section{Conclusion}
This paper established the basic theory of hereditary polytopes.  One should pursue these ideas further by considering some of the following problems, which have been brought to light by our work.   

As a first example, one could examine if there exist hereditary polytopes whose $i$-faces are all themselves non-regular hereditary polytopes ($i \geq 3$).  In other words, given any hereditary polytope $\pp$, can another hereditary polytope be built which has $\pp$ as its facets?  This question is open even when $\pp$ is of rank 3. 

In this paper we considered polytopes where the automorphism group of each facet is a subgroup of the full automorphism group of the polytope.  
It would also be of interest to study ``chirally hereditary" polytopes, that is, those polytopes which are not hereditary, but have the property that each rotational symmetry of a facet extends to a global symmetry.   For example, an interesting class of such objects is the chiral polytopes with regular facets - which includes all chiral maps.

Additionally, it would be of interest to investigate the idea of geometrically hereditary polytopes.  For example in $\mathbb{E}^3$, can one classify the $i$-face transitive geometrically hereditary polyhedra, that is, those with symmetry group inheriting all isometries of their polygonal faces?  The rhombic dodecahedron is an example of a 2-face transitive geometrically hereditary polyhedron. (For a survey on related questions for convex polyhedra see also \cite{mar}.)

\begin{acknowledgement}
A substantial part of this article was written while we visited the Fields Institute for extended periods of time during the Thematic Program on Discrete Geometry and Application in Fall 2011. We greatly appreciated the hospitality of the Fields Institute and are very grateful for the support we received. Mark Mixer was Fields Postdoctoral Fellow in Fall 2011. Egon Schulte was also supported by NSF-grant DMS--0856675, and Asia Ivi\'{c} Weiss by NSERC. Finally, we wish to thank Peter McMullen and Barry Monson for helpful comments.
\end{acknowledgement}


\vskip.5in

\begin{thebibliography}{99.}
\bibitem{magma} Bosma, W., Cannon, C., Playoust C.: The Magma algebra system I: The user language. 
J. Symbolic Comput. {\bf 24}, 235--265 (1997)

\bibitem{bjs} Breda D'Azevedo, A., Jones, G.A, Schulte, E.: Constructions of chiral polytopes of small rank. Canadian J. Math. {\bf 63}, 1254--1283 (2011)

\bibitem{Conder} Conder, M.: Regular maps and hypermaps of Euler characteristic -1 to -200. 
J. Combin. Theory Ser. B.  {\bf 99}, 455--459 (2009)
with associated lists available at http://www.math.auckland.ac.nz/$\sim$conder/.

\bibitem{chp} Conder, M., Hubard, I., Pisanski, T.: Constructions for chiral polytopes. J. London Math. Soc. (2) {\bf 77}, 115--129 (2008)

\bibitem{crsp} Coxeter, H.S.M.: Regular skew polyhedra in 3 and 4 dimensions
and their topological analogues.  Proc. London Math. Soc. (2) {\bf 43}, 33--62 (1937) (Reprinted with amendments in {\em Twelve Geometric Essays},
Southern Illinois University Press (Carbondale, 1968), 76--105.)

\bibitem{cox2} Coxeter, H.S.M.: Regular and semi-regular polytopes, I, 
Math.\ Z. {\bf 46}, 380--407 (1940).  In: Sherk, F.A., McMullen, P., Thompson, A.C.,  
Weiss, A.I. (eds.) Kaleidoscopes: Selected
Writings of H.S.M. Coxeter, pp 251--278. Wiley-Interscience, New York (1995)

\bibitem{clm} Coxeter, H.S.M., Longuet-Higgins, M.S., Miller, J.C.P.: Uniform Polyhedra. Phil. Trans. Roy. Soc. London Ser. A {\bf 246}, 401--450 (1954)

\bibitem{cm}
Coxeter, H.S.M., Moser, W.O.J.: Generators and Relations for Discrete Groups, volume~14 of
  Ergebnisse der Mathematik und ihrer Grenzgebiete [Results in Mathematics and
  Related Areas].
\newblock Springer-Verlag, Berlin, fourth edition, (1980)

\bibitem{dan} Danzer, L.: Regular Incidence-Complexes and Dimensionally Unbounded Sequences of Such, I. In: Rosenfeld, M., Zaks, J. (eds.), North-Holland Mathematics Studies, Volume 87, pp 115--127. North-Holland, (1984)

\bibitem{gw} Graver, J.E., Watkins, M.E.: Locally Finite, Planar, Edge-Transitive Graphs. Memoirs of the Amer. Math. Soc. 126 (1997)

\bibitem{Gr2}
Gr\"unbaum, B.: Regularity of graphs, complexes and designs. Probl\`{e}mes combinatoires et th\'{e}orie des graphes. Coll. Int. CNRS, {\bf 260}, 191--197 (1977) 

\bibitem{grsh} Gr\"unbaum, B., Shephard, G.C.:  Tilings and Patterns. Freeman \& Co., New York (1987)

\bibitem{Hartley} Hartley, M.I.: The atlas of small regular abstract polytopes. Period. Math. Hungar. {\bf 53}, 149--156 (2006)
(http://www.abstract-polytopes.com/atlas/)

\bibitem{hub} Hubard, I.: Two-orbit polyhedra from groups. Europ. J. Comb. {\bf 31}, 943--960 (2010) 

\bibitem{hubsch} Hubard, I., Schulte, E.: Two-orbit polytopes. in preparation.

\bibitem{hubow} Hubard, I., Orbanic, A., Weiss, A.I.: Monodromy groups and self-invariance. Canadian J. Math.  {\bf 61}, 1300--1324 (2009)

\bibitem{luwe} Lu\v{c}i\'{c} Z., Weiss, A.I.: Regular polyhedra in hyperbolic three-space, Coxeter-Festschrift (Part~III), Mitt. Math. Sem. Giessen {\bf 165\/}, 237--252 (1984)

\bibitem{mar} Martini, H.:  A hierarchical classification of euclidean polytopes with regularity properties.  In: Bisztriczky, T., McMullen, P., Schneider, R., Ivi\'c Weiss, A. (eds.) Polytopes:  Abstract, Convex and Computational, pp. 71--96. NATO ASI Series C {\bf 440},  Kluwer, Dordrecht (1994)

\bibitem{arp} McMullen, P., Schulte, E.: Abstract Regular Polytopes. Encyclopedia of Math.Appl., vol. 92, Cambridge University Press, Cambridge (2002)

\bibitem{monsch}
Monson, B., Schulte, E.: Semiregular polytopes and amalgamated C-groups. Advances Math. {\bf 229}, 2767--2791 (2012)

\bibitem{OPW} Orbani\'{c}, A., Pellicer, D., Weiss, A.I.: Map operations and k-orbit maps. J. Comb. Theory A.  {\bf 117}, 411--429 (2012)

\bibitem{pell} Pellicer, D.: A construction of higher rank chiral polytopes. Discrete Math. {\bf 310}, 1222--1237 (2010)

\bibitem{esext} Schulte, E.: Extensions of regular complexes.  In: Baker, C.A., Batten, L.M. (eds.) Finite
Geometries. Lecture Notes Pure
Applied Mathematics {\bf 103}, pp. 289--305. Marcel Dekker, New York (1985)

\bibitem{spacfi} Schulte, E.: Space-fillers of higher genus. J. Comb.Theory A. {\bf 68}, 438--453 (1994)

\bibitem{SW2} Schulte, E., Weiss, A.I.: Chirality and projective linear groups.  Discrete Math. {\bf 131}, 221--261 (1994)

\bibitem{banff} Schulte, E., Weiss, A.I.: Problems on polytopes, their groups, and realizations. Period. Math. Hungar. {\bf 53}, 231--255 (2006)

\bibitem{stw} \v{S}iran, J., Tucker, T.W., Watkins, M.: Realizing finite edge-transitive orientable maps. J. Graph Theory {\bf 37}, 1--34 (2001)


\end{thebibliography}
\end{document}